\documentclass[reqno,10pt]{amsart}
\pdfoutput=1

\usepackage[utf8]{inputenc}
\usepackage[english]{babel}
\usepackage{amsfonts}
\usepackage{amssymb}
\usepackage{bbm}
\usepackage{tikz}
\usepackage[right=2cm,bottom=2cm,footskip=30pt,left=2.5cm]{geometry}
\usepackage[autostyle]{csquotes}
\MakeOuterQuote{"}
\usepackage{color}
\usepackage{mathrsfs}
\usepackage{mathtools,leftidx}
\usepackage{enumitem}
\usepackage{amsmath}
\usepackage{mathrsfs}
\usepackage{accents}
\usepackage{scalerel}
\usepackage[most]{tcolorbox}
\usepackage{aligned-overset}
\setlist[enumerate,2]{label={(\theenumi.\theenumii)},ref={(\theenumi.\theenumii)},leftmargin=1cm}
\setlist[enumerate,1]{label={(\theenumi)},ref={(\theenumi)},leftmargin=1cm}

\usetikzlibrary{svg.path,arrows,calc}
\tikzset{>=latex'}

\usepackage[toc]{appendix}
\usepackage{etoolbox}
\cslet{blx@noerroretextools}\empty
\usepackage[maxbibnames=9,giveninits=true,style=alphabetic,sorting=nyt,backend=biber]{biblatex}
\DeclareFieldFormat{titlecase:title}{\MakeSentenceCase*{#1}}
\def\restrict#1{\raise-.5ex\hbox{\ensuremath|}_{#1}}

\usepackage[colorlinks, citecolor=citegreen, linkcolor=refred,urlcolor=blue, unicode,psdextra,hypertexnames=false]{hyperref}
\usepackage{cleveref}
\usepackage{orcidlink}
\expandafter\def\csname ver@etex.sty\endcsname{3000/12/31}

\makeatletter
 \def\author@andify{
 \nxandlist {\unskip{\kern.3cm} \penalty-2}
 {\unskip {\kern.3cm} \penalty-2}
 {\unskip {\kern.3cm} \penalty-2}}
\makeatother

\usepackage{accents}
\usepackage{autonum}
\usepackage{gasymbols}
\usepackage{nicefrac}

\definecolor{citegreen}{rgb}{0,0.3,0}
\definecolor{refred}{rgb}{0.5,0,0}

\makeatletter
\renewcommand{\paragraph}{%
    \@startsection{paragraph}{4}{1em}%
    {1.25ex \@plus 1ex \@minus .2ex}%
    {-.3em}%
    {\normalfont\normalsize\bfseries}%
}
\makeatother

\def\mathring#1{\accentset{\circ}{#1}}
\usepackage[equation=section,capitalcref]{theorems}

\let\oldemail\email
\let\email\relax
\def\email#1{\oldemail{\href{mailto:#1}{\textcolor{black}{#1}}}}

\title[An isoperimetric characterization of a new ADM-like mass]{An isoperimetric characterization of a new ADM-like mass\\for $C^0$-asymptotically flat manifolds}

\author[L.~Benatti]{Luca Benatti \orcidlink{0000-0002-4685-7443}}
\address{L.~Benatti, University of Vienna,
Oskar-Morgenstern-Platz 1, 1090, Vienna, Austria}
\email{luca.benatti@univie.ac.at}

\author[M.~Fogagnolo]{Mattia Fogagnolo \orcidlink{0000-0002-5933-1344}}
\address{M.~Fogagnolo, Universit\`a di Padova, via Trieste 63, 35121 Padova (PD), Italy}
\email{mattia.fogagnolo@unipd.it}

\renewcommand{\ncapa}{\mathfrak{c}}


\newcommand{\sml}[1]{%
  \mathchoice%
    {#1}
    {#1}
    {\scaleto{(#1)}{5pt}}
    {\scaleto{(#1)}{4pt}}
}

\newcommand{\An}{\mathrm{An}}
\begin{document}
\begin{abstract}
    We prove that the notion of mass recently introduced by Mazurowski and Yao for continuous metrics coincides with Huisken's isoperimetric mass on smooth  Riemannian $3$-manifolds with nonnegative scalar curvature whose metric has sharp $C^0$-asymptotically flat behavior. As a consequence, we deduce a Riemannian Penrose inequality for the mass parameter of asymptotically Schwarzschildian $3$-manifolds.
\end{abstract}

\maketitle

\noindent MSC (2020): 
53C21, 
53E10, 
38C35, 
49J45. 

\smallskip

\noindent \underline{\smash{Keywords}}: concepts of mass, isoperimetric mass, isocapacitary mass, nonnegative scalar curvature.

\section{Introduction}
A geometric quantity of pivotal importance in mathematical general relativity is the ADM mass. Originally introduced in \cite{arnowitt_coordinate_1961}, it describes the total mass of an isolated gravitational system. Two fundamental results reflecting its significance are the positive mass theorem and its refinement, the Riemannian Penrose inequality. The first asserts that the ADM mass is nonnegative when the scalar curvature is nonnegative. In dimension $3$, this was proved by Schoen and Yau in \cite{schoen_proof_1979}; their argument was later extended to dimensions $3\leq n\leq 7$ in \cite{schoen_variational_1989}, and more recently to all dimensions in \cite{brendle_dimension_2026}. The second gives a lower bound for the ADM mass in terms of the area of the outermost minimal boundary. In dimension $3$, it was proved independently by Huisken and Ilmanen \cite{huisken_inverse_2001} and Bray \cite{bray_proof_2001}. The higher-dimensional case $3\leq n\leq 7$ was established in \cite{bray_riemannian_2009}, while the result without dimensional restrictions was recently obtained in \cite{bi_riemannian_2026}.

The ADM mass is defined through first derivatives of the metric and requires sufficiently strong asymptotic control to be well defined and coordinate invariant \cite{bartnik_mass_1986,chrusciel_invariant_1988}. These limitations have motivated the search for alternative notions of mass expressed in terms of geometric or analytic quantities that remain meaningful under weaker regularity assumptions.

A prominent example is the isoperimetric mass, introduced by G. Huisken in \cite{huisken_isoperimetric_2006} in dimension $3$. For a bounded region $\Omega$, the corresponding quasi-local mass measures the (normalized) deficit between its volume and that of a Euclidean ball with the same boundary area, and the global mass is obtained by considering these quantities along exhaustions of the manifold; see \Cref{subsec:isomass} below. Since it is formulated entirely in terms of volumes and areas, the isoperimetric mass is coordinate invariant by construction.

The isoperimetric mass provides a significant surrogate to the ADM mass: it remains meaningful for merely continuous Riemannian metrics and, whenever the ADM mass is well defined and nonnegative, the two notions coincide. Building on fundamental contributions of Huisken \cite{huisken_marston_2009} and Jauregui--Lee \cite{jauregui_lower_2019}, this equivalence was established under optimal assumptions by the authors together with L. Mazzieri in \cite{benatti_isoperimetric_2025}. However, once the assumptions guaranteeing the nonnegativity of the ADM mass are dropped, the two notions may differ. Indeed, the isoperimetric mass is always nonnegative on asymptotically flat manifolds \cite{jauregui_note_2024}; see also \cite{antonelli_positive_2026}.

A genuinely low-regularity ADM-type notion of mass was first introduced by P. Burkhardt-Guim \cite{burkhardt-guim_adm_2024}, who established its well-posedness using Ricci flow. More recently, L. Mazurowski and X. Yao \cite[Theorem 3]{mazurowski_positive_2026} proposed a related notion in dimension $3$, which is well posed on $C_\tau$-asymptotically flat manifolds with nonnegative scalar curvature, for every $\tau>1/2$. Here, $C_\tau$-asymptotic flatness means that the end admits coordinates $x$ identifying it with the complement of a Euclidean ball and such that $\abs{g-\delta}=O(\abs{x}^{-\tau})$. For every $\psi \in C^{\infty}_c(0,1)$, they consider the quantity
\begin{equation}\label{eq:MYmass}
\ma_{\mathrm{MY}}=
\lim_{r\to\infty} \int_{M} K_{\psi}^{ij}(x,r)(g_{ij}-\delta_{ij}) \dif x,
\end{equation}
where $K^{ij}_\psi(x,r)$ is a smooth symmetric tensor supported in the annulus $\set{r \leq \abs{x} \leq 4r}$. This tensor is built so that an integration-by-parts argument recovers the classical ADM mass whenever the metric is smooth and satisfies the corresponding first-order asymptotic decay.

The existence of the limit and its independence of $\psi$ is established in \cite{mazurowski_positive_2026} by relating it with the asymptotic behavior of the Hawking-type quantity introduced by Agostiniani--Mazzieri--Oronzio \cite{agostiniani_greens_2024} along the level sets of Green kernels. This route also provides a positive mass theorem for this quantity that holds in fact for continuous metrics with nonnegative scalar curvature in a suitable approximate sense. The masses proposed in \cite{burkhardt-guim_adm_2024} and in \cite{mazurowski_positive_2026} are shown to coincide in dimension $3$ when $\tau > 2/3$.

The natural question about the equivalence of the isoperimetric mass and the low-regularity ADM mass of Mazurowski--Yao then arises. We answer it affirmatively in the smooth case. 

\begin{theorem}\label{thm:intro-1}
Let $(M,g)$ be a smooth, connected, complete, $C_\tau$-asymptotically flat Riemannian $3$-manifold with one end, where $\tau>1/2$, and with compact, possibly empty, boundary. Assume that the scalar curvature of $(M,g)$ is nonnegative and that $\partial M$ is minimal. Suppose moreover that no other compact minimal surface is contained in $M$. Then
\begin{equation}\label{eq:harmonic-isoperimetric-equivalence}
\ma_{\mathrm{MY}}=\ma_{\iso}.
\end{equation}
\end{theorem}

The proof of \Cref{thm:intro-1} relies on a family of nonlinear potential-theoretic quantities associated with the $p$-Laplace equation. Given a bounded domain $\Omega$ and its $p$-capacitary potential $u_p$, Agostiniani--Mantegazza--Mazzieri--Oronzio \cite{agostiniani_riemannian_2025} introduced a family of quantities monotone along the level sets of $u_p$, which were later named $p$-Hawking masses in \cite{benatti_nonlinear_2023}. The relevance of this construction is reinforced by the fact that, as $p\to1$, the functions $w_p=-(p-1)\log u_p$ provide a natural approximation of Huisken--Ilmanen's weak inverse mean curvature flow as shown in \cite{moser_inverse_2007,kotschwar_local_2009,benatti_proper_2026}. For this reason, we refer to the level-set flow generated by $w_p$ as the $p$-inverse mean curvature flow, or $p$-IMCF for short. 

At the level of the associated geometric quantities, this picture is completed by the convergence of the $p$-Hawking masses to the classical Hawking mass as $p\to1$, proved by the first-named author jointly with A. Pluda and M. Pozzetta \cite{benatti_fine_2026}. Combining this approximation result with the asymptotic comparison arguments for $p$-Hawking masses developed in \cite{benatti_nonlinear_2023} then shows that the Hawking mass of any outward minimizing domain with connected boundary is bounded from above by $\ma_{\mathrm{MY}}$. Applying Jauregui--Lee's machinery \cite{jauregui_lower_2019}, this yields $\ma_{\iso}\leq\ma_{\mathrm{MY}}$. For the reverse inequality, deducing from \cite{mazurowski_positive_2026} that $\ma_{\mathrm{MY}}$ can be computed as the limit at infinity of the $2$-Hawking mass, we employ another asymptotic comparison argument to bound it from above by a suitable capacitary version of the isoperimetric mass introduced by J. Jauregui \cite{jauregui_adm_2024}. By the first author's work \cite{benatti_equivalence_2025}, this quantity coincides with the isoperimetric mass, and hence $\ma_{\iso} \geq \ma_{\mathrm{MY}}$.

The asymptotic comparison arguments just mentioned are refined versions of those developed in \cite{benatti_isoperimetric_2025,benatti_nonlinear_2023}. The key improvement comes from $W^{1,q}$ asymptotics for the $p$-IMCF under merely $C^0$-asymptotic flatness, which ultimately rely on the gradient estimate of Kinnunen--Zhou \cite{kinnunen_local_1999}. This gradient estimate also plays a central role in the work of Mazurowski and Yao in the linear case $p=2$. Its availability for $p>1$ is precisely what makes the $p$-IMCF framework preferable here to working directly with the weak IMCF.

\bigskip

Finally, we focus on asymptotically flat $3$-manifolds whose deviation from the Euclidean metric admits a Schwarzschild-type leading term. More precisely, we say that $(M,g)$ is asymptotically Schwarzschildian if, in asymptotically flat coordinates,
\begin{equation}
        g -\delta =\frac{2\ma}{\abs{x}}\delta+o\left(\frac{1}{\abs{x}}\right),
    \end{equation}
for some $\ma \in \R$. Substituting the asymptotic expansion of $g$ into \cref{eq:MYmass} directly yields $\ma = \ma_{\mathrm{MY}}$. Therefore, \cref{thm:intro-1} also identifies the Schwarzschild mass parameter with the isoperimetric mass, see \cref{thm:intro-2}. Coupling it with the isoperimetric Penrose inequality \cite[Theorem 1.3]{benatti_isoperimetric_2025} yields 
that 
 \begin{equation}
 \label{eq:penrose-schwar}
        \sqrt{\frac{\abs{N}}{16\pi}}\leq\ma
\end{equation}
for any connected component $N$ of the outermost minimal boundary of the manifold, with equality holding on spatial Schwarzschild only. A weaker version of this result was previously obtained in \cite[Theorem 1.5]{benatti_isoperimetric_2025} by a different argument under the additional assumption $\ma\geq0$. 

We can interpret \cref{eq:penrose-schwar} and its rigidity statement also in connection with a recently proved conjecture of M. Gromov \cite[Sect 3.11]{gromov_four_2023}. It can be rephrased by claiming that $\ma = 0$ implies isometry with flat $\R^3$. In particular, it consists in a special case of the rigidity statement of the Penrose inequality above. See \cite{mazurowski_positive_2026} and \cite{you_gromovs_2026} for proofs of Gromov's statement. 

\medskip

The paper is organized as follows. In \Cref{sec:preliminaries}, we recall the $p$-IMCF, the nonlinear capacities, the elliptic regularization, and the associated notions of Hawking and isocapacitary masses.  In \Cref{sec:main-results}, we determine the asymptotic behavior of the $p$-IMCF and work out the asymptotic comparison arguments leading to \cref{thm:intro-1} and its consequences in the asymptotically Schwarzschildian setting.

\paragraph{Consequences and other questions} We already observed that, in dimension $3$, the masses introduced in \cite{burkhardt-guim_adm_2024} and \cite{mazurowski_positive_2026} are shown to coincide when $\tau>2/3$; see \cite[Proposition 23]{mazurowski_positive_2026}. \cref{thm:intro-1} then implies that the mass introduced in \cite{burkhardt-guim_adm_2024} also coincides with Huisken's isoperimetric mass for smooth metrics under the same decay assumption at infinity \cite{jauregui_adm_2024}. Moreover, part of our argument relies on the equivalence between the isoperimetric mass and Jauregui's isocapacitary mass, established by the first-named author in \cite{benatti_equivalence_2025}. More generally, the same equivalence holds for the whole family of $p$-isocapacitary masses introduced in \cite{benatti_nonlinear_2023}. Thus, in the smooth $3$-dimensional setting with $\tau>2/3$, all these notions of mass fit into a common framework and, in fact, agree. In particular, this gives an affirmative answer to \cite[Question 1]{burkhardt-guim_adm_2024} in this setting. It would be interesting to understand whether this equivalence persists under the sharp decay assumption $\tau>1/2$, and ultimately at the level of continuous asymptotically flat metrics. Such an extension would provide a unified interpretation of several $C^0$-compatible notions of total mass.

\paragraph{AI usage} Artificial intelligence tools were used to correct grammar and spelling in parts of the text, to assist with routine computations, and to help identify potential issues in the arguments. They were not used to produce any of the proofs.

\paragraph{Acknowledgments} This research was funded in part by the Austrian Science Fund (FWF) [grant DOI \href{https://www.fwf.ac.at/en/research-radar/10.55776/EFP6}{10.55776/EFP6}]. For open access purposes, the author has applied a CC BY public copyright license to any author-accepted manuscript version arising from this submission.

M. F. is supported by the STARS project “DEFORM” of the University of Padova and by the project "GIANTS" funded by INdAM. 

The authors are members of the INDAM–GNAMPA.

\section{Preliminaries on the \texorpdfstring{$p$}{p}-inverse mean curvature flow and notions of mass}\label{sec:preliminaries}

\begin{center}
\begin{tcolorbox}[    
    enhanced,
    frame empty,
    colback=gray!10,
    sharp corners,
    boxrule=0pt,
    boxsep=0pt,
    left=10pt,
    right=10pt,
    top=7pt,
    bottom=7pt,
    before skip=8pt,
    after skip=8pt,
    borderline west={2pt}{0pt}{gray!55},
    before upper={\setlength{\parindent}{12pt}}]
Throughout the paper, $(M,g)$ will be a smooth, connected, complete, noncompact Riemannian $3$-manifold with one end. Its boundary, if nonempty, is assumed to be compact and minimal. We also assume that $M$ contains no other compact minimal surfaces. 
\smallskip

$\Omega\subset M$ will be a compact subset such that $M \smallsetminus \Omega$ is a domain whose boundary is a $C^{1,1}$-hypersurface and each connected component of $\partial M$ is either contained in or disjoint from $\Omega$. 
\end{tcolorbox}
\end{center}

We point out that the absence of compact minimal surfaces in $M$, together with the existence of a mean-convex exhaustion, implies that $H_2(M,\partial M;\mathbb Z)=\set{0}$; see \cite[Lemma 2.11]{benatti_isoperimetric_2025}. This topological condition is needed in order to apply Geroch-type monotonicity formulas such as \cref{thm:Geroch-monotonicity}. In the asymptotically flat setting, the existence of the required exhaustion follows from Xu's construction of a proper inverse mean curvature flow \cite{xu_isoperimetry_2024}. The precise asymptotic assumptions on $(M,g)$ will be specified whenever they are needed.

\begin{definition}\label{def:asymptotically_flat}
Let $(M,g)$ be a Riemannian $3$-manifold. Given $\tau>0$, we say that $(M,g)$ is $C_\tau$-asymptotically flat if the following conditions are satisfied:
\begin{enumerate}
    \item there exist a compact set $K\subseteq M$, a radius $R>0$, and a diffeomorphism
    \begin{equation}
        x=(x^1,x^2,x^3):M\smallsetminus K\longrightarrow\R^3\smallsetminus\set{\abs{x}\leq R};
    \end{equation}
    the components of $x$ are called asymptotically flat coordinates;
    \item in the asymptotically flat coordinate chart, the metric takes the form
    \begin{equation}
        g=g_{ij}\,\dd x^i\otimes\dd x^j=(\delta_{ij}+\zeta_{ij})\,\dd x^i\otimes\dd x^j,
    \end{equation}
    where $\abs{x}^{\tau}\abs{\zeta_{ij}} = O(1)$ as $\abs{x} \to +\infty$ for all $i,j =1,2,3$.
\end{enumerate}
We say that $(M,g)$ is asymptotically flat if (1) holds and $\zeta_{ij}$ in (2) satisfies $\abs{\zeta_{ij}}=o(1)$ for $i,j=1,2,3$.
\end{definition}

In this section, we collect the main notions and preliminary results that will be used throughout the sequel. We first recall the definition of the $p$-inverse mean curvature flow, together with its main properties, and the nonlinear notions of Hawking mass and isocapacitary mass \cite{agostiniani_greens_2024, agostiniani_riemannian_2025, benatti_nonlinear_2023}. We then introduce an elliptic regularization of the flow and discuss the corresponding convergence and regularity properties. Finally, we introduce the Mazurowski--Yao mass.

\subsection{\texorpdfstring{$p$}{p}-capacitary potentials and the \texorpdfstring{$p$}{p}-IMCF} For $p>1$, the $p$-capacitary potential of $\Omega$ is (if it exists) the unique weak solution $u_p$ of
\begin{equation}\label{eq:p-capacitary}
    \begin{cases}
        \Delta_pu_p&=&0 & \text{in }M\smallsetminus \Omega,\\
        u_p&=&1 & \text{on }\Sigma,\\
        u_p(x)&\longrightarrow& 0 & \text{as }\abs{x} \to\infty.
    \end{cases}
\end{equation}
where $\Delta_p f \coloneqq \div( \abs{\nabla f}^{p-2} \nabla f)$ is the $p$-Laplacian.

The $p$-inverse mean curvature flow issuing from $\Omega$ is defined by
\begin{equation}\label{eq:p-IMCF_moser_rule}
w_p\coloneqq -(p-1)\log u_p.
\end{equation}
A direct computation shows that $w_p$ is the proper weak solution of
\begin{equation}\label{eq:p-IMCF}
\begin{cases}
\Delta_pw_p&=&\abs{\nabla w_p}^p & \text{in }M\smallsetminus \Omega,\\
w_p&=&0 & \text{on }\Sigma,\\
w_p(x)&\longrightarrow&+\infty & \text{as }\abs{x} \to\infty.
\end{cases}
\end{equation}

The terminology $p$-inverse mean curvature flow is motivated by its formal analogy with the weak inverse mean curvature flow, which corresponds to the limiting case $p=1$. The latter was introduced by Huisken and Ilmanen in \cite{huisken_inverse_2001} and was employed in their proof of the Riemannian Penrose inequality. The connection between the inverse mean curvature flow and nonlinear potential theory goes back to Moser \cite{moser_inverse_2007,moser_inverse_2008}, who observed that $w_p$ provides an elliptic approximation of the weak inverse mean curvature flow as $p\to1^+$. More precisely, Moser proved that any locally uniform limit of such functions is a weak inverse mean curvature flow. This approximation procedure was subsequently developed by Kotschwar and Ni \cite{kotschwar_local_2009} and by Mari, Rigoli, and Setti \cite{mari_1h-flow_2022} (cf. \cite{benatti_proper_2026} for the updated result), who extended it to more general settings. More recently, the relation between the two flows was strengthened in \cite{benatti_fine_2026}, where the convergence of the $p$-IMCF to the weak inverse mean curvature flow was proved in strong Sobolev spaces and the convergence of their level sets and associated geometric quantities was analyzed.

We will denote 
\begin{equation}
    E^{(p)}_t = \set{w_p\leq t} \qquad  \Sigma_t^{(p)} =\partial E^{(p)}_t \qquad \text{for every } t\in[0,+\infty)
\end{equation}
Since $\abs{\nabla w_p}\neq 0$ on $\Sigma$ by the Hopf maximum principle, $E^{(p)}_0= \overline{\Omega}$ and $\Sigma_0^{(p)}=\Sigma$.

By the regularity theory for the $p$-Laplace equation,
\begin{equation}
u_p,w_p\in C_{\loc}^{1,\alpha}(M\smallsetminus\operatorname{Int}\Omega)
\end{equation}
for some $\alpha\in(0,1)$, and both functions are smooth outside their critical set. Moreover,
\begin{equation}
\abs{\nabla w_p}^{p-1}\in W_{\loc}^{1,2}(M\smallsetminus \overline{\Omega}).
\end{equation}

As a consequence of coarea formula, $\Sigma^{(p)}_t \cap \set{\abs{\nabla w_p}=0}$ is $\Hff^2$-negligible for almost every level. By \cite{benatti_fine_2026}, almost every level set $\Sigma^{(p)}_t$ also has $L^2$-integrable weak second fundamental form in the sense of varifolds. Moreover, the mean curvature of $\Sigma_t^{(p)}$ can be expressed as
\begin{align}
    \H = \abs{\nabla w_p}- (p-1) \frac{\ip{\nabla \abs{\nabla w_p}, \nabla w_p}}{\abs{\nabla w_p}^2}
\end{align}
at almost every point, for almost every $t \in [0,+\infty)$. Thus, the usual inverse mean curvature relation $\H=\abs{\nabla w_1}$ is formally recovered as $p\to1^+$.

We conclude this part by recalling the main existence result for asymptotically flat manifolds \cite[Proposition B.2]{benatti_isoperimetric_2025}.
\begin{proposition}\label{lem:existence_and_Li_Yau}
Let $(M,g)$ be an asymptotically flat Riemannian $3$-manifold and let $1<p<3$. For any $\Omega\subseteq M$, there exists a unique solution $w_p$ of \cref{eq:p-IMCF}. Moreover, there are a constant $\kst>1$ and a radius $R>0$ such that
\begin{equation}
(3-p) \log \abs{x} -\kst \leq w_p \leq (3-p) \log \abs{x} + \kst
\end{equation}
whenever $\abs{x}\geq R$.
\end{proposition}

\subsection{\texorpdfstring{$p$}{p}-capacity and its evolution}

For a compact set $K\subseteq M$, its (normalized) $p$-capacity is defined by
\begin{equation}\label{eq:def-p-capacity}
\ncapa_p(K)\coloneqq\inf\set{\frac{1}{4\pi}\left(\frac{p-1}{3-p}\right)^{p-1}\int_M\abs{\nabla v}^p\dif\mu_g\st v\in C_c^\infty(M),\ v\geq1\text{ on }K}.
\end{equation}
If $u_{p}$ is the $p$-capacitary potential associated with $\Omega$, then
\begin{equation}
\ncapa_p(\Sigma)=\frac{1}{4\pi}\left(\frac{p-1}{3-p}\right)^{p-1}\int_{M\smallsetminus \Omega}\abs{\nabla u_{p}}^p\dif\mu_g.
\end{equation}

 A fundamental property of the $p$-capacity is that it is monotone with respect to the standard inclusion of sets. Moreover, the $p$-capacity of $\Sigma_t^{(p)}$ is exponentially growing with respect to the parameter $t$. This follows from \cite[Lemma 3.8]{holopainen_nonlinear_1990} and the relation $w_p=-(p-1)\log u_p$, since $e^{t/(p-1)}u_p$ is the $p$-capacitary potential of $\Sigma_t^{(p)}$.
 \begin{lemma}\label{lem:expontial_growth}
     Let $(M,g)$ be a Riemannian $3$-manifold and let $w_p$ be the $p$-IMCF issuing from $\Omega$ as in \cref{eq:p-IMCF}. Then
    \begin{equation}\label{eq:exponential_growth}
    \ncapa_p(\Sigma^{(p)}_t) = \ee^t \ncapa_p(\Sigma).
    \end{equation}
 \end{lemma}

The identity \cref{eq:exponential_growth} is the nonlinear potential-theoretic counterpart of the exponential area growth along the inverse mean curvature flow. Indeed, we have the following result, that is \cite[Theorem 1.2]{fogagnolo_minimising_2022}.
\begin{lemma} \label{lem:p-capacity_to_area}
Let $(M,g)$ be an asymptotically flat manifold. If $\Omega\subseteq M$ is outward minimizing, then
    \begin{equation}
        \lim_{p \to 1^+}\ncapa_p(\Sigma)  = \frac{\abs{\Sigma}}{4 \pi}.
    \end{equation}
\end{lemma}

\subsection{The \texorpdfstring{$\varepsilon$}{ε}-regularization}\label{sec:epsilon-regularization}
We recall the elliptic regularization of the $p$-IMCF. Let $w_p$ be the solution of \cref{eq:p-IMCF} and let $D\Subset M$ be a smooth bounded open set such that $E_T^{(p)} \Subset D$ for some $T>0$. Let $w_{p,\varepsilon}$ be the solution of
\begin{equation}\label{eq:epsilon-regularized-p-imcf}
    \begin{cases}
        \Delta_{p,\varepsilon}\left(e^{-\frac{w_{p,\varepsilon}}{p-1}}\right)&=&0 & \text{in }D\setminus\Omega,\\
        w_{p,\varepsilon}&=&0 & \text{on }\partial\Omega,\\
        w_{p,\varepsilon}&=&w_p & \text{on }\partial D,
    \end{cases}
\end{equation}
where $\abs{X}_\varepsilon = \sqrt{\abs{X}^2+\varepsilon^2}$ and $\Delta_{p,\varepsilon}f \coloneqq \div( \abs{\nabla f}_{\varepsilon}^{p-2} \nabla f)$. Equivalently, one can build $w_{p,\varepsilon}$ by taking a $(p,\varepsilon)$-harmonic function $u_{p,\varepsilon}$ and using the usual relation $w_{p,\varepsilon} =-(p-1) \log u_{p,\varepsilon}$. The function $u_{p,\varepsilon}$ is the regularized solution classically used to prove regularity of $p$-harmonic functions.

We set 
\begin{equation}
    E_t^{(\varepsilon)} = \set{w_{p,\varepsilon} \leq t} \qquad \Sigma^{(\varepsilon)}_t  = \partial E_t^{(\varepsilon)}\qquad \text{for every }t<T.
\end{equation}
By the maximum principle, the sets $E_t^{(\varepsilon)}$ are uniformly contained in $D$. Moreover, since $w_{p,\varepsilon}$ solves a uniformly elliptic equation, it is smooth in the interior of $D \smallsetminus \Omega$. In particular, Sard's theorem gives that $\Sigma^{(\varepsilon)}_t$ is regular for almost every $t \in (0,T)$.

The main advantage of \cref{eq:epsilon-regularized-p-imcf} is that its regularity estimates are uniform as $\varepsilon\to0^+$. As a consequence, $(w_{p,\varepsilon})_{\varepsilon>0}$ provides a good approximation to $w_p$. More precisely, the properties of $w_p$ recalled at the beginning of the section actually come from this approximation:
\begin{itemize}
    \item $ w_{p,\varepsilon}\to w_p$ in $C_{\loc}^{1, \alpha}(D\smallsetminus\operatorname{Int}\Omega)$  for some $\alpha<1$ by \cite{dibenedetto_c1alpha_1983,tolksdorf_regularity_1984,lewis_regularity_1983};
    \item $\abs{\nabla w_{p,\varepsilon}}^{p-1} \rightharpoonup \abs{\nabla w_p}^{p-1}$ weakly in $W_{\loc}^{1,2}(D\smallsetminus\Omega)$ by \cite{lou_singular_2008} (see also Appendix C in \cite{benatti_monotonicity_2022});
    \item  $\Sigma^{(\varepsilon)}_t\to \Sigma_t^{(p)}$ (up to a subsequence) in the sense of curvature varifold for almost every $t \in (0,T)$  by \cite{benatti_fine_2026}.
\end{itemize}

\medskip

A consequence of the uniform convergence and the uniform containment of the levels is that $q$-capacities are continuous, namely 
\begin{equation}\label{eq:convergence-q-cap_eps-regularized}
    \lim_{\varepsilon \to 0^+} \ncapa_q(\Sigma^{(\varepsilon)}_t) = \ncapa_q(\Sigma_t^{(p)})
\end{equation}
for every $t$ outside an at most countable set. This result follows from a more general statement.
\begin{lemma}\label{lem:q-capacity-convergence}
    Let $(M,g)$ be a Riemannian manifold and let $(f_k)_{k \in \N}$ be a sequence of nonnegative functions on $M$ converging locally uniformly to a function $f_{\infty}$. Let $E^{(k)}_t = \set{f_k \leq t}$. Assume that there exists a bounded open set $D$ and a $T>0$ such that $E_T^{(k)} \subseteq D\Subset M$ for every $k$. Then,
    \begin{equation}
        \lim_{k \to +\infty} \ncapa_q(E^{(k)}_t) = \ncapa_q(E_t^\infty)
    \end{equation}
    for every $t<T$ outside an at most countable set for every $q>1$.
\end{lemma}
\begin{proof}
    We only need to prove the statement when $M$ is $q$-nonparabolic. Indeed, otherwise the $q$-capacity of every compact set is zero and the conclusion is immediate. Since $E^{(\infty)}_s\subseteq E^{(\infty)}_t$ whenever $s\leq t$, the monotonicity of the $q$-capacity implies that the function $t\longmapsto\ncapa_q(E^{(\infty)}_t)$ is nondecreasing on $(0,T)$. Consequently, it is continuous outside an at most countable set.
    Fix a continuity point $t<T$ of $s\mapsto\ncapa_q(E^{(\infty)}_s)$ and let $\delta>0$ be such that $t+\delta<T$. Since $f_k\to f_\infty$ uniformly on $D$, there exists $k_0\in\N$ such that 
    \begin{equation} 
        \abs{f_k-f_{\infty}}<\delta\qquad\text{on } D 
    \end{equation} 
    for every $k\geq k_0$. It follows that $E^{(\infty)}_{t-\delta}\subseteq E_t^{(k)}\subseteq E^{(\infty)}_{t+\delta} $ for every $k \geq k_0$. Applying the monotonicity of the $q$-capacity, we obtain $\ncapa_q(E^{(\infty)}_{t-\delta})\leq\ncapa_q(E_t^{(k)})\leq\ncapa_q(E^{(\infty)}_{t+\delta}) $    for every $k\geq k_0$. Therefore, 
    \begin{equation} 
        \ncapa_q(E^{(\infty)}_{t-\delta})\leq\liminf_{k\to+\infty}\ncapa_q(E_t^{(k)})\leq\limsup_{k\to+\infty}\ncapa_q(E_t^{(k)})\leq\ncapa_q(E^{(\infty)}_{t+\delta}). 
    \end{equation} 
    Letting $\delta\to0^+$ and using the continuity of $s\mapsto\ncapa_q(E_s^{(\infty)})$ at $t$, we conclude that 
    \begin{equation} 
        \lim_{k\to+\infty}\ncapa_q(E_t^{(k)})=\ncapa_q(E^{(\infty)}_t). \qedhere
    \end{equation}
\end{proof}

\begin{remark}\label{rmk:convergence_with_moving_metrics}
    The result still holds if we replace the metric $g$ with a sequence of metrics $g_k$ uniformly converging to $g$ on $M$. Indeed, for any compact $K\Subset M$ one can take an admissible competitor $v$ in \cref{eq:def-p-capacity} and show that
    \begin{equation}
        \frac{(1- \varepsilon_k)^{\frac{n}{2}}}{(1+ \varepsilon_k)^{\frac{q}{2}}}  \int_{M} \abs{\nabla v}^q_{g} \dif \mu_{g}\leq \int_{M} \abs{\nabla v}^q_{g_k} \dif \mu_{g_k}\leq \frac{(1+ \varepsilon_k)^{\frac{n}{2}}}{(1- \varepsilon_k)^{\frac{q}{2}}} \int_{M} \abs{\nabla v}^q_{g} \dif \mu_{g},
    \end{equation}
     where $\varepsilon_k$ tends to zero as $k\to+ \infty$ and is such that $\abs{g- g_k}\leq \varepsilon_k$ on $M$. Taking the infimum we get
     \begin{equation}
         \frac{(1- \varepsilon_k)^{\frac{n}{2}}}{(1+ \varepsilon_k)^{\frac{q}{2}}} \ncapa_q(K;g) \leq \ncapa_q(K;g_k) \leq \frac{(1+ \varepsilon_k)^{\frac{n}{2}}}{(1- \varepsilon_k)^{\frac{q}{2}}} \ncapa_q(K;g),
     \end{equation}
     where $\ncapa_q(K;g)$ denotes the $q$-capacity with respect to the metric $g$. The conclusion follows by taking the limit.
\end{remark}

\medskip 

In general, varifold convergence implies that $L^2$-norms of the mean curvature and the second fundamental form are lower semicontinuous. However, as $\Sigma^{(\varepsilon)}_t\to \Sigma^{(p)}_t$, continuity holds. This is shown in \cite[Proposition 5.1]{benatti_fine_2026}; we recall here the statement tailored for our purposes.
\begin{proposition}\label{prop:convergence-meancurvature}
    Let $w_p$ be a solution to \cref{eq:p-IMCF} and $w_{p,\varepsilon}$ the solution to \cref{eq:epsilon-regularized-p-imcf}. Let $(\varepsilon_k)_{k \in \N}$  be a vanishing sequence such that $\Sigma^{(\varepsilon_k)}_t\to \Sigma^{(p)}_t$ for almost every $t \in (0,T)$ in the sense of curvature varifold. Then,
    \begin{align}
        \label{eq:epsilon-mean-curvature-energy-convergence}
        \lim_{k\to+\infty} \int_{\Sigma_t^{(\varepsilon_k)}} \H^2\dif\sigma&=\int_{\Sigma^{(p)}_t}\H^2\dif\sigma\\
        \label{eq:epsilon-second-fundamental-form-energy-convergence}
        \lim_{k\to+\infty} \int_{\Sigma_t^{(\varepsilon_k)}} \abs{\h}^2\dif\sigma &=\int_{\Sigma^{(p)}_t} \abs{\h}^2 \dif\sigma
    \end{align}
    for almost every $t \in [0,T)$.
\end{proposition}

A direct consequence of this result is the following Gauss--Bonnet-type theorem for level sets of $p$-IMCF:

\begin{proposition}[{\cite[Corollary 5.3]{benatti_fine_2026}}]\label{prop:Gauss-Bonnet}
    Under the same hypotheses as \Cref{prop:convergence-meancurvature}, we have
    \begin{equation}\label{eq:gb}
        \int_{\Sigma^{(p)}_t} \sca^\top \dif \sigma = \lim_{k \to +\infty} 4 \pi \chi(\Sigma_t^{(\varepsilon_k)})
    \end{equation}
    for almost every $t \in (0,T)$. 
\end{proposition}
\begin{remark}
    The quantity $\sca^\top$ denotes the induced scalar curvature. On almost every level set $\Sigma_t^{(p)}$, it can be defined away from the critical set of the potential via the Gauss equation.
\end{remark}

\begin{remark}
    Observe that, under the assumptions of \cref{thm:intro-1}, the surfaces $\Sigma^{(\varepsilon_k)}_t$ are connected and smooth for almost every $t\in(0,T)$. Therefore, the right-hand side of \cref{eq:gb} does not exceed $8\pi$.
\end{remark}

\subsection{The isocapacitary and isoperimetric masses}
\label{subsec:isomass}
The quasi-local $p$-isocapacitary mass of $\Omega$ is defined by
\begin{equation}\label{eq:quasi-local-isocapacitary-mass}
\ma_{\iso}^{(p)}(\Omega)\coloneqq\frac{1}{2p\pi\ncapa_p(\partial\Omega)^{\frac{2}{3-p}}}\left[\abs{\Omega}-\frac{4\pi}{3}\ncapa_p(\partial\Omega)^{\frac{3}{3-p}}\right].
\end{equation}
The global $p$-isocapacitary mass of $(M,g)$ is
\begin{equation}
\ma_{\iso}^{(p)}\coloneqq\sup_{\set{\Omega_j}}\limsup_{j\to+\infty}\ma_{\iso}^{(p)}(\Omega_j),
\end{equation}
where the supremum is taken over all exhaustions $\set{\Omega_j}$ of $M$.

For $p=1$, the mass reduces to the isoperimetric mass introduced by Huisken in \cite{huisken_isoperimetric_2006}, i.e.
\begin{equation}
\ma_{\iso}(\Omega)\coloneqq\frac{2}{\abs{\partial \Omega}}\left[\abs{\Omega} - \frac{\abs{\partial \Omega}^{\frac{3}{2}}}{6 \sqrt{\pi}}\right].
\end{equation}

The first author recently proved that those notions of mass are actually equivalent on asymptotically flat manifolds with nonnegative scalar curvature. 
\begin{theorem}[{\cite{benatti_equivalence_2025}}]\label{thm:equivalence_luca}
    Let $(M,g)$ be an asymptotically flat Riemannian $3$-manifold with nonnegative scalar curvature. Then, $\ma^{\sml{p}}_{\iso} = \ma_{\iso}$ for all $p\in [1,3)$.
\end{theorem}
\subsection{The \texorpdfstring{$p$}{p}-Hawking masses}

Let $w_p$ be a $p$-IMCF issuing from $\Omega$ as in \cref{eq:p-IMCF}. The $p$-Hawking mass of $\Sigma= \partial \Omega$ is defined by
\begin{equation}\label{eq:p-Hawking}
\ma_H^{(p)}(\Sigma)\coloneqq\frac{\ncapa_p(\Sigma)^{\frac{1}{3-p}}}{8\pi}\left[4\pi+\int_\Sigma\frac{\abs{\nabla w_{p}}^2}{(3-p)^2}\dif\sigma-\int_\Sigma\frac{\abs{\nabla w_{p}}}{3-p}\H\dif\sigma\right].
\end{equation}
For $p=1$, it reduces to the classical Hawking mass
\begin{equation}
\ma_H(\Sigma)\coloneqq\sqrt{\frac{\abs{\Sigma}}{16 \pi }}\left[1-\int_{\Sigma}\frac{\H^2}{16 \pi}\dif\sigma\right].
\end{equation}

\begin{theorem}[Monotonicity of the $p$-Hawking mass]\label{thm:Geroch-monotonicity}
    Let $(M,g)$ be a Riemannian $3$-manifold with nonnegative scalar curvature. Assume that $H_2(M, \partial M; \mathbb{Z})= \set{0}$. Let $w_p$ be the solution of \cref{eq:p-IMCF} issuing from $\Omega$ with connected boundary. Then, the function $t\mapsto\ma_H^{(p)}(\Sigma^{(p)}_t)$ admits a monotone nondecreasing representative in $\operatorname{BV}_{\loc}(0,+\infty)$. Moreover, for almost every $t \in (0,+\infty)$
    \begin{align}
    \frac{\dif}{\dif t}\ma_H^{(p)}(\Sigma^{(p)}_t)
    &\geq \frac{\ncapa_p(\Sigma^{(p)}_t)^{\frac{1}{3-p}}}{8\pi(3-p)}
    \left[4\pi-\int_{\Sigma^{(p)}_t}\frac{\sca^\top}{2}\dif\sigma+\int_{\Sigma^{(p)}_t}\frac{\abs{\mathring{\h}}^2}{2}+\frac{\sca}{2}
    +\frac{\abs{\nabla^\top\abs{\nabla w_p}}^2}{\abs{\nabla w_p}^2} +\frac{5-p}{p-1}
    \left(\frac{\abs{\nabla w_p}}{3-p}-\frac{\H}{2}\right)^2\dif\sigma\right],
    \end{align}
    where $\sca$ is the scalar curvature of $(M,g)$, $\sca^\top$ is the scalar curvature of $\Sigma^{(p)}_t$, $\mathring{\h}$ is its traceless second fundamental form, and $\nabla^\top$ denotes the tangential gradient.
\end{theorem}

\begin{proof}
    One can compute the derivatives using the Leibniz rule and \cite[Theorem B.1(7)]{benatti_fine_2026} for $p>1$ and \cite[Theorem 4.1]{benatti_fine_2026} for $p=1$.
\end{proof}

This theorem was originally proved by Huisken--Ilmanen for $p=1$ in \cite{huisken_inverse_2001} and by Agostiniani--Mantegazza--Mazzieri--Oronzio for $p>1$ in \cite{agostiniani_riemannian_2025}.

\smallskip
In \cite{benatti_isoperimetric_2025} and \cite{benatti_nonlinear_2023}, we proved that each $p$-isocapacitary mass provides a natural upper bound for the corresponding $p$-Hawking mass on a large class of surfaces. In the case $p=1$, the converse implication also holds: controlling the Hawking mass along a sufficiently rich family of surfaces is enough to control the isoperimetric mass. This key result, due to Jauregui and Lee \cite{jauregui_lower_2019}, will be used in the following form.

\begin{theorem}
\label{thm:jl-Hawking}
    Let $(M,g)$ be an asymptotically flat Riemannian $3$-manifold with a possibly empty, compact, smooth boundary. Suppose that the boundary consists of minimal surfaces and no other compact minimal surface is contained in $M$. If $\ma_{H}(\partial \Omega) \leq \ma < +\infty$ for every outward minimizing $\Omega\supseteq \partial M$ with connected $C^{1,1}$-boundary, then $\ma_{\iso} \leq \ma$.
\end{theorem}

We refer the reader to \cite[Theorem 4.2]{benatti_equivalence_2025} and the discussion below for the proof of the statement in this form.

\subsection{Mazurowski--Yao mass}
    Let $(M,g)$ be a $C_\tau$-asymptotically flat Riemannian $3$-manifold, $\tau>\frac12$, with nonnegative scalar curvature.
    For every $\psi\in C^\infty_c(0,1)$ nonnegative and $\psi \not\equiv 0 $, we set $\eta_\psi \in C^\infty_c(1,4)$ as
    \begin{equation}\label{eq:eta-psi}
        \eta_\psi(s)\coloneqq\frac{1}{2s^3}\int_{s/2-1}^{s-1}\psi(t)\dif t.
    \end{equation}
    
    We also define $\An(r)=\set{r\leq\abs{x}\leq4r}$, and set
    \begin{align}\label{eq:local_mass_functional}
        L_g^\psi(r)&\coloneqq\frac{1}{r}\int_{\An(r)}\left(\frac{1}{\abs{x}}\eta_\psi\left(\frac{\abs{x}}{r}\right)+\frac{1}{r}\eta_\psi'\left(\frac{\abs{x}}{r}\right)\right)\delta^{ij}(g_{ij} -\delta_{ij})\dif x\\
        &\qquad +\frac{1}{r}\int_{\An(r)}\left(\frac{1}{\abs{x}}\eta_\psi\left(\frac{\abs{x}}{r}\right)-\frac{1}{r}\eta_\psi'\left(\frac{\abs{x}}{r}\right)\right)(g_{ij} -\delta_{ij})\frac{x^ix^j}{\abs{x}^2}\dif x.
    \end{align}

    \begin{proposition}\label{prop:wellposedness}
        Let $(M,g)$ be a $C_\tau$-asymptotically flat Riemannian $3$-manifold, $\tau>\frac12$, with nonnegative scalar curvature. Then,
        \begin{equation}
            \lim_{t \to +\infty} \ma^{(2)}_H (\Sigma^{(2)}_t) =  \left(3\pi\int_0^1\frac{\psi(s)}{(1+s)^2}\dif s\right)^{-1} \lim_{r \to +\infty} L^\psi_g(r),
        \end{equation}
        in particular the right-hand side does not depend on $\psi$.
    \end{proposition}
    \begin{proof}
        The proof is based on the argument of \cite[Theorem 25]{mazurowski_positive_2026}, where a similar equivalence is established using a suitably normalized harmonic kernel. The relevant ingredient in that argument is not that the harmonic function is a Green kernel on the whole manifold, but rather its normalization at infinity. This motivates introducing
        \begin{equation}\label{eq:normalized_armonic}
            \hat{u}_2 = \frac{u_2}{\ncapa_2(\partial \Omega)}.
        \end{equation}
        Indeed, with this normalization, the same argument as in \cite[Proposition 26]{mazurowski_positive_2026} yields
        \begin{equation}\label{eq:zzzrefinedasymptoticbehaviour}
            \norm{r\hat{u}_2(r x) - \frac{1}{\abs{x}}}_{W^{1,q}(\An(1))} = O(r^{-\beta}) \qquad \text{for some } \beta > \frac{1}{2}.
        \end{equation}
    
        In analogy with \cite[Definition 13]{mazurowski_positive_2026}, we build from $\hat{u}_2$ the function
        \begin{equation}\label{eq:Dr}
            D(r) = 2\pi \int_0^1\frac{\psi(s)}{1+s}\dif s+\int \theta(r,\hat{u}_2)\abs{\nabla \hat{u}_2}^3\dif\mu_g,
        \end{equation}
        where
        \begin{equation}
            \theta(r,t) = \frac{1}{t^3}\left[\frac{1}{2}\psi\left(\frac{1}{2rt}-1\right)-\psi\left(\frac{1}{rt}-1\right)\right].
        \end{equation}
        We first prove that $r \mapsto rD(r)$ is monotone in our setting. Let $r(t) = \ncapa_2(\Sigma^{(2)}_t)$. Since $r(t) = \ncapa_2(\partial \Omega)\ee^t$ by \cref{lem:expontial_growth}, the normalization in \cref{eq:normalized_armonic} gives            $\Sigma^{(2)}_t = \set{\hat{u}_2 = r(t)^{-1}}$. Moreover, on $\Sigma^{(2)}_t$ we have $\abs{\nabla w_2} = r(t)\abs{\nabla \hat{u}_2}$. Consequently, the definition of the $2$-Hawking mass gives
        \begin{equation}\label{eq:zzz2Hawing-F-relation}
            \begin{split}
                8\pi\ma_H^{(2)}(\Sigma^{(2)}_t)
                &= r(t)\left(4\pi+\int_{\Sigma^{(2)}_t}\abs{\nabla w_2}^2\dif\sigma-\int_{\Sigma^{(2)}_t}\H\abs{\nabla w_2}\dif\sigma\right) \\
                &= 4\pi r(t)-r(t)^2\int_{\Sigma^{(2)}_t}\H\abs{\nabla\hat{u}_2}\dif\sigma+r(t)^3\int_{\Sigma^{(2)}_t}\abs{\nabla\hat{u}_2}^2\dif\sigma.
            \end{split}
        \end{equation}
        Accordingly, define
        \begin{equation}\label{eq:equivalence_F_mH2}
            F(r) = 4\pi r-r^2\int_{\set{\hat{u}_2 = \frac{1}{r}}}\H\abs{\nabla\hat{u}_2}\dif\sigma+r^3\int_{\set{\hat{u}_2 = \frac{1}{r}}}\abs{\nabla\hat{u}_2}^2\dif\sigma.
        \end{equation}
        By \cref{eq:zzz2Hawing-F-relation}, we have $F(r(t)) = 8\pi\ma_H^{(2)}(\Sigma^{(2)}_t)$. Since $t(r)=\log r-\log\ncapa_2(\partial\Omega)$ is increasing, \cref{thm:Geroch-monotonicity} implies that $r\mapsto F(r)$ is monotone. The identity in \cite[Proposition 19]{mazurowski_positive_2026} relating $D$ and $F$ then yields
        \begin{equation}\label{eq:equivalence_F_rDr}
            rD(r) = \int_0^1\frac{\psi(v)}{(1+v)^2}\left[\int_1^2\frac{F((1+v)rs)}{s^3}\dif s\right]\dif v,
        \end{equation}
        showing that $r \mapsto rD(r)$ is monotone as well. Moreover,
        \begin{align}\label{eq:Dtom}
            \lim_{r\to+\infty}rD(r)
            &= 8\pi\int_0^1\frac{\psi(v)}{(1+v)^2}\left[\int_1^2\frac{1}{s^3}\dif s\right]\dif v\lim_{t\to+\infty}\ma_H^{(2)}(\Sigma^{(2)}_t) \\
            &= \left(3\pi\int_0^1\frac{\psi(s)}{(1+s)^2}\dif s\right)\lim_{t\to+\infty}\ma_H^{(2)}(\Sigma^{(2)}_t).
        \end{align}
    
        Finally, inspecting the last part of the proof of \cite[Theorem 25]{mazurowski_positive_2026}, one sees that the strategy only requires a harmonic function with the asymptotic estimate \cref{eq:zzzrefinedasymptoticbehaviour}. Therefore, the same argument applies to $\hat{u}_2$ and gives
        \begin{equation}\label{eq:zzzMYequivalencetheorem25}
            \lim_{r\to+\infty}rD(r) = \lim_{r\to+\infty}L^\psi_g(r).
        \end{equation}
        Combining \cref{eq:Dtom,eq:zzzMYequivalencetheorem25} concludes the proof.
    \end{proof}

    The above result allows for the following definition.

    \begin{definition}\label{def:mildADM}
        Let $(M,g)$ be a $C_\tau$-asymptotically flat Riemannian $3$-manifold, $\tau>\frac12$, with nonnegative scalar curvature. Given any $\psi \in C^\infty_c(0,1)$ nonnegative with $\psi \not\equiv 0$, the Mazurowski-Yao mass is defined as
        \begin{equation}
            \ma_{\mathrm{MY}} \coloneqq \left(3\pi\int_0^1\frac{\psi(s)}{(1+s)^2}\dif s\right)^{-1} \lim_{r \to +\infty} L^\psi_g(r).
        \end{equation}
    \end{definition}

\section{Proof of the main results}\label{sec:main-results}
In this section, we prove the main results of the paper. 
We first determine the asymptotic behavior of the $p$-IMCF under $C^0$-asymptotic flatness. This allows us to carry over the technique developed in \cite{benatti_nonlinear_2023} to the present, more general setting, and accordingly establish the asymptotic comparisons between the different $p$-Hawking-type quantities introduced above. We then recall the harmonic mass of Mazurowski--Yao and we prove \Cref{thm:intro-2}. In conclusion, we identify this harmonic mass with the limit of the $2$-Hawking mass and combine it with the comparison between Hawking and isoperimetric mass to prove \Cref{thm:intro-1}.
\subsection{\texorpdfstring{$p$}{p}-harmonic functions on asymptotically flat manifolds}

The purpose of this subsection is to establish the asymptotic behavior of solutions $w_p$ to \cref{eq:p-IMCF} on asymptotically flat manifolds. We collect the main conclusions in the following statement.
\begin{theorem}\label{thm:asymptotic_behaviour}
Let $(M,g)$ be an asymptotically flat Riemannian $3$-manifold and $1<p<3$. Let $\Omega\subseteq M$ and let $w_p$ be the unique solution to \cref{eq:p-IMCF}. Then,
\begin{enumerate}
    \item $w_p = (3-p) \log \abs{x} - \log \ncapa_p(\partial \Omega) + o(1)$ as $\abs{x} \to +\infty$;
    \item for every $1<q<3$ we have
    \begin{equation}
        \lim_{t\to+\infty}\frac{\ncapa_q(\Sigma^{(p)}_t)^{3-p}}{\ncapa_p(\Sigma^{(p)}_t)^{3-q}}=1;
    \end{equation}
    \item there exists a divergent increasing sequence $\set{t_j}_{j\in\N}$ such that
    \begin{equation}
        \lim_{j\to+\infty}\int_{\Sigma^{(p)}_{t_j}}\abs{\nabla w_p}^2\dif\sigma_g=4\pi(3-p)^2.
    \end{equation}
\end{enumerate}
\end{theorem}

The argument is based on a blow-down procedure (cf. \cite{huisken_inverse_2001,benatti_asymptotic_2024,mazurowski_positive_2026}). We write $\An=\set{x \in \R^3: a < |x| <b}$ for $0<a<b<+\infty$.  For sufficiently large $r>0$, we can consider the family of metrics $g_r$ and functions $w_p^r$ defined by
\begin{equation}
g_r(x)=g(rx),\qquad w_p^r(x)=w_p(rx)-(3-p)\log r.
\end{equation}
Since $g$ is asymptotically flat, $g_r$ converges uniformly to the Euclidean metric on $\set{\abs{x} >a}$ for every $a>0$. Moreover, the two-sided estimates in \Cref{lem:existence_and_Li_Yau} give locally uniform bounds for $w_p^r$. To conclude, it only remains to ensure a gradient bound. In the present setting, the rescaled metrics converge only in $C^0$, and hence the uniform gradient estimate used in \cite[Lemma 2.9]{benatti_nonlinear_2023} is not available. We replace it with the following result of Kinnunen and Zhou \cite{kinnunen_local_1999}. 

\begin{theorem}[Kinnunen--Zhou estimate]\label{thm:kinnunen-zhou}
    Let $(M,g)$ be an $n$-dimensional manifold, let $1<p<q<+\infty$, and let $u \in W^{1,p}(U)$ be $p$-harmonic in some domain $U$. Then $u\in W_{\loc}^{1,q}(U)$. More precisely, for every pair of relatively compact open sets $V\Subset U$, there exists a constant $C>0$ such that
    \begin{equation}\label{eq:riemannian-kinnunen-zhou-estimate}
        \norm{\nabla u}_{L^q(V)}\leq C\norm{u}_{L^q(U)}.
    \end{equation}
    The constant depends only on $n$, $p$, $q$, the set $V$ and $g$. This dependence is stable under locally uniform convergence of the metrics: if $g_k \to g$ locally uniformly on $U$ and $g$ is a metric, then the corresponding constants can be chosen uniformly in $k$.
\end{theorem}

\begin{proof}
    In \cite[Theorem 1.4]{kinnunen_local_1999}, choose 
    \begin{equation}
        A_g = (\det g)^{\frac1p}(g^{ij}).
    \end{equation}
    It is easy to show that $A_g$ satisfies the conditions of the theorem. The constant only depends on the ellipticity constant of $g$ and the $\operatorname{VMO}$ data of $A_g$. If $g_k\to g$ locally uniformly and $g$ is a metric, the ellipticity constants of $g_k$ are controlled by that of $g$ for $k$ sufficiently large. The same holds for the $\operatorname{VMO}$ data of $A_{g_k}$.  
\end{proof}

We are now ready to prove the Sobolev convergence from which \Cref{thm:asymptotic_behaviour} will follow.

\begin{proposition}\label{prop:strong_W1p_asymptotics}
Let $(M,g)$ be an asymptotically flat Riemannian $3$-manifold and $1<p<3$. Let $w_p$ be the $p$-IMCF issuing from $\Omega$ as in \cref{eq:p-IMCF}. Then, as $r\to+\infty$,
\begin{equation}
    w_p^r\to (3-p)\log\abs{x}-\log\ncapa_p(\partial\Omega) \qquad \text{strongly in } W_{\loc}^{1,q}(\R^3\smallsetminus \set{0})
\end{equation}
for every $1<q<+\infty$.
\end{proposition}

\begin{proof}
Consider two annuli $\An \Subset \An'$. We will prove the equivalent statement for $u_p^r$ given by $w_p^r = -(p-1) \log u_p^r$. The bounds in \Cref{lem:existence_and_Li_Yau} ensure
\begin{equation}
    \kst^{-1} \abs{x}^{-\frac{3-p}{p-1}} \leq u_p^r\leq \kst \abs{x}^{-\frac{3-p}{p-1}}.
\end{equation}
Moreover, $u_p^r$ is a positive $p$-harmonic function on $\An'$ with respect to the metric $g_r$. By \Cref{thm:kinnunen-zhou}, the gradient bound
\begin{equation}\label{eq:zzzkinnunen-zhou_consequence}
    \norm{\nabla u^r_p}_{L^q(\An)} \leq \kst \norm{u^r_p}_{L^q(\An')}
\end{equation}
holds for every $p<q<+\infty$. Choosing $q>3$ and using a diagonal argument, there exists a sequence $u_p^{r_k} \to u_p^\infty$ in $C^0_{\loc}(\R^3 \smallsetminus \set{0})$ and $u^{r_k}_p \rightharpoonup u_p^\infty$ weakly in $W^{1,q}_{\loc}(\R^3 \smallsetminus \set{0})$. Since $g_r \to \delta$ uniformly sufficiently far from the origin, $u_p^\infty$ is a $p$-harmonic function on $\R^3 \smallsetminus \set{0}$ satisfying 
\begin{equation}
    \kst^{-1} \abs{x}^{-\frac{3-p}{p-1}} \leq u_p^\infty\leq \kst \abs{x}^{-\frac{3-p}{p-1}}
\end{equation}
for all $x\in \R^3 \smallsetminus \set{0}$. \cite[Proposition 3.3]{benatti_asymptotic_2024} yields $u_p^\infty= \gamma \abs{x}^{-\frac{3-p}{p-1}}$ for some positive $\gamma$. \Cref{lem:q_capacity_convergence} and \Cref{rmk:convergence_with_moving_metrics} imply that $\gamma^{p-1}= \ncapa_p(\partial \Omega)$. Indeed,
\begin{equation}
    \gamma^{p-1} = \ncapa_p(\set{u_p^\infty \geq 1}; \delta) = \lim_{k\to +\infty} \ncapa_p(\set{u_p^{r_k}\geq 1} ; g_{r_k}) =\ncapa_p(\partial \Omega).
\end{equation}
In particular, the whole sequence $u^r_p$ converges locally uniformly and weakly in $W^{1,q}_{\loc}$ to $u_p^\infty$.

\bigskip
To conclude, it is enough to show that $\nabla u^r_p \to \nabla u$ in $L^{p}_{\loc}(\R^3 \smallsetminus \set{0})$. Indeed, the statement follows by H\"older's inequality if $q<p$ and by the interpolation inequality and \cref{eq:zzzkinnunen-zhou_consequence} for $q>p$.

For $r>1$, denote by $X_r$ the vector fields defined as
\begin{equation}
    X_r^k(x,\xi) = (A_r^{ij}(x) \xi_j \xi_i)^{\frac{p-2}{2}}A^{kl}_r(x)\xi_l \qquad \text{where }  A^{ij}_r(x) = (\det g_r)^{\frac{1}{p}} g_r^{ij}.
\end{equation}
We agree that $A^{ij}_{\infty} = I$. Since $u^r_p$ is $p$-harmonic, 
\begin{equation}\label{eq:zzz-weak-p-in-coordinates}
    \int_{\An} X_r^k(x, \dd u^r_p) \partial_k \varphi \dif x =0
\end{equation}
for every $\varphi \in C^\infty_c(\An)$. Take now $\eta \in C^\infty_c(\An)$ and let $v^r_p = u^r_p - u_p^\infty$. The function $\eta^p v^r_p$ is not smooth, but still sufficiently regular to replace $\varphi$ in \cref{eq:zzz-weak-p-in-coordinates}. Therefore, we get
\begin{equation}\label{eq:zzz-integration-bypart}
\begin{split}
    \int_{\An}\eta^p [X_r^k(x, \dd u^{r}_p)-X^k_r(x, \dd u_p^\infty)] \partial_k v_p^r \dif x &= -\int_{\An} \eta^p[X_r^k(x, \dd u^\infty_p)-X^k_\infty(x, \dd u_p^\infty)] \partial_k v_p^r \dif x\\&\qquad  - p \int_{\An} \eta^{p-1} v^r_p [X_r^k(x, \dd u^r_p)-X^k_\infty(x, \dd u_p^\infty)]\partial_k \eta \dif x
\end{split}
\end{equation} 

Since $X_r(x, \xi) \to X_\infty(x,\xi)$ uniformly on compact subsets of $\An$ for fixed $\xi$, $ \nabla u_p^\infty$ is uniformly bounded on $\An$ and $\nabla v_p^r$ is uniformly bounded in $L^p(\An)$, we get
\begin{equation}\label{eq:zzz-first-integral-convergence}
    \lim_{r \to +\infty } \int_{\An} \eta^p[X_r^k(x, \dd u^\infty_p)-X^k_\infty(x, \dd u_p^\infty)] \partial_k v_p^r \dif x =0
\end{equation}
On the other hand, by definition we have
\begin{equation}
   \abs{X_r (x, \xi)}^2_{g_r} = \det g_r \, (g_r^{kl} \xi_k \xi_l)^{p-1}.
\end{equation}
Hence, since $v_p^r \to 0$ uniformly and the gradients $\nabla u_p^r$ are uniformly bounded in $L^p(\An)$, we also get
\begin{equation}\label{eq:zzz-second-integral-convergence}
    \lim_{r\to +\infty } \int_{\An} \eta^{p-1} v^r_p [X_{r}^k(x, \dd u^r_p)-X^k_\infty(x, \dd u_p^\infty)]\partial_k \eta \dif x =0.
\end{equation}
Taking the limit in \cref{eq:zzz-integration-bypart} and plugging in \cref{eq:zzz-first-integral-convergence,eq:zzz-second-integral-convergence} we conclude that
\begin{equation}\label{eq:zzz-final-strong-convergence}
     \lim_{r \to +\infty } \int_{\An}\eta^p [X_r^k(x, \dd u^r_p)-X^k_r(x, \dd u_p^\infty)] \partial_k v_p^r \dif x = 0.
\end{equation}
\smallskip

We now use \cref{eq:zzz-final-strong-convergence} to prove the convergence of the gradients in $L^p_{\loc}$. By the quantitative monotonicity inequalities for uniformly elliptic $p$-Laplace type operators\footnote{See, for instance, \cite[Chapter 12]{lindqvist_notes_2019}. For $p\geq 2$, the relevant inequality is (I), whereas the case $1<p<2$ follows along the same lines as (VII). We have omitted the additional $1$, since $\abs{\nabla u^\infty_p}$ never vanishes, so no indeterminate form can arise.}, we have
\begin{equation}\label{eq:zzz-uniform-ellipticity-bigp}
    \kst[c]\abs{ \nabla v^r_p}^p  \leq [X_r^k(x, \dd u^{r}_p)-X^k_r(x, \dd u_p^\infty)] \partial_k v_p^r
\end{equation}
for $p\geq 2$ and
\begin{equation}\label{eq:zzz-uniform-ellipticity-smallp}
    \kst[c] \abs{\nabla v^r_p}^2 \left[\abs{\nabla u_p^r}+ \abs{\nabla u ^\infty_p}\right]^{p-2}   \leq [X_r^k(x, \dd u^r_p)-X^k_r(x, \dd u_p^\infty)] \partial_k v_p^r
\end{equation}
for $1<p<2$.

If $p\geq 2 $, \cref{eq:zzz-final-strong-convergence} and \cref{eq:zzz-uniform-ellipticity-bigp} immediately imply that $\nabla v^r_p \to 0$ strongly in $L^p$ on the region where $\eta=1$,  which yields the conclusion since $\eta $ and $\An$ are arbitrary. If $1<p<2$, the conclusion follows in the same way from \cref{eq:zzz-final-strong-convergence} and \cref{eq:zzz-uniform-ellipticity-smallp} once one observes that
\begin{align}
    \int_{\set{\eta=1}}\abs{\nabla v_p^r}^p\dif x&\leq\left[\int_{\set{\eta=1}}\left(\abs{\nabla u_{p}^r}+\abs{\nabla u_{p}^\infty}\right)^{p-2}\abs{\nabla v_p^r}^2\dif x\right]^{\frac{p}{2}}\left[\int_{\set{\eta=1}}\left(\abs{\nabla u_{p}^r}+\abs{\nabla u_{p}^\infty}\right)^p\dif x\right]^{\frac{2-p}{2}}
\end{align}
and the $L^p$ norms of $\nabla u_p^r$ and $\nabla u^\infty_p$ are uniformly bounded.
\end{proof}

Item (1) in \Cref{thm:asymptotic_behaviour} is part of the proof of \Cref{prop:strong_W1p_asymptotics}. Items (2) and (3) are the content of the following two corollaries.

\begin{corollary}\label{lem:q_capacity_convergence}
    Let $(M,g)$ be an asymptotically flat Riemannian $3$-manifold and let $p,q\in(1,3)$. Let $w_p$ be the $p$-IMCF issuing from $\Omega$ as in \cref{eq:p-IMCF}. Then,
    \begin{equation}
        \lim_{t\to+\infty}\frac{\ncapa_q(\Sigma^{(p)}_t)^{3-p}}{\ncapa_p(\Sigma^{(p)}_t)^{{3-q}}}=1.
    \end{equation}
\end{corollary}

\begin{proof}
    It is enough to combine \Cref{prop:strong_W1p_asymptotics} with \Cref{lem:q-capacity-convergence}. Indeed, denote $E^{(r)}_t= \set{w_p^r\leq t} $ and $\Sigma^{(r)}_t = \partial E^{(r)}_t$. Letting $t= (3-p) \log r$, we have
    \begin{equation}
         \ee^{\frac{q-3}{3-p}t} \ncapa_q(\Sigma^{(p)}_{t}) =  \ncapa_q(\Sigma^{(r)}_0;g_r) 
    \end{equation}
    by the scaling property of the $q$-capacity. Then, the conclusion follows since 
    \begin{equation}
        \lim_{r\to +\infty} \ncapa_q(\Sigma^{(r)}_0;g_r)  = \ncapa_q(\Sigma^{(\infty)}_0;\delta) = \ncapa_p(\partial \Omega)^{\frac{3-q}{3-p}}. \qedhere
    \end{equation}
\end{proof}

\begin{corollary}\label{lem:asymptotic-gradient-level-sequence}
Let $(M,g)$ be an asymptotically flat Riemannian $3$-manifold and let $p\in(1,3)$. Let $w_p$ be the $p$-IMCF issuing from $\Omega$ as in \cref{eq:p-IMCF}. Then, there exists a divergent increasing sequence $\set{t_j}_{j\in\N}$ such that
\begin{equation}\label{eq:asymptotic-gradient-level-sequence}
\lim_{j\to+\infty}\int_{\Sigma^{(p)}_{t_j}}\abs{\nabla w_p}^2 \dif\sigma_g=4\pi(3-p)^2.
\end{equation}
\end{corollary}

\begin{proof}
    Let $0<a<b<\infty$ and $w_p^\infty=(3-p) \log \abs{x} - \log \ncapa_p(\partial \Omega)$. Denote $E^{(r)}_t=\set{w_p^{r}\leq t}$ and $\Sigma^{(r)}_t= \partial E^{(r)}_t$. \Cref{prop:strong_W1p_asymptotics} implies that
    \begin{equation}
        \lim_{r\to +\infty} \int_{E^{(r)}_b \smallsetminus E^{(r)}_a} \abs{\nabla w^r_p}^3 \dif \mu_{g_r} = \int_{E^{(\infty)}_b \smallsetminus E^{(\infty)}_a} \abs{\nabla w_p^\infty}^3 \dif x
    \end{equation}
    Indeed, $\dd \mu_{g_r} \to \dd x $ uniformly, and the characteristic functions and the gradient strongly converge in every $L^q$, $q< \infty$. On the other hand, the coarea formula yields
    \begin{align}
        \int_{E^{(\infty)}_b \smallsetminus E^{(\infty)}_a} \abs{\nabla w_p^\infty}^3\dif x=\int_a^b\int_{\Sigma^{(\infty)}_s}\abs{\nabla w_p^\infty}^2\dif\sigma\dif s=4\pi(3-p)^2(b-a).
    \end{align}
    The function
    \begin{equation}
        s \mapsto \int_{\Sigma^{(r)}_s} \abs{\nabla w_p^r}^2 \dif \sigma_{g_r}
    \end{equation}
    belongs to $W^{2,1}_{\loc}$ by \cite[Proposition B.1]{benatti_fine_2026} and therefore it admits a continuous representative. Applying the mean value theorem to the continuous representative, for every $r>0$ we can find $s_r \in(a,b)$ such that 
    \begin{equation}
        \abs{\frac{1}{b-a}   \int_{E^{(r)}_b \smallsetminus E^{(r)}_a}
        \abs{\nabla w^r_p}^3 \dif \mu_{g_r} - \int_{\Sigma_{s_r}^{(r)}} \abs{\nabla w_p^r}^2 \dif \sigma_{g_r}} \leq \frac{1}{r}.
    \end{equation}
    Choosing $r_j=\ee^{j(b-a)/(3-p)}$ and denoting $t_j = s_{r_j} + j(b-a)$, we get the desired sequence.    
\end{proof}

\subsection{Asymptotic comparisons between \texorpdfstring{$p$}{p}-Hawking masses}
We now have all the basic tools needed to prove our main results. We introduce an auxiliary modification of the $p$-Hawking mass \cref{eq:p-Hawking}, that is
\begin{equation}
    \tilde{\ma}_H^{(p)}(\Sigma)\coloneqq \frac{\ncapa_p(\Sigma)^{\frac{1}{3-p}}}{4 \pi (3-p)} \left[ 4 \pi - \int_{\Sigma}\frac{\abs{\nabla w_p}^2}{(3-p)^2}\dif\sigma\right].
\end{equation}
This quantity is generally different from the $p$-Hawking mass only for $p>1$. Moreover, to our knowledge the derivative of $\tilde{\ma}_H^{(p)}$ along the level sets of $w_p$ does not have a definite sign in general, under nonnegative scalar curvature. The monotonicity result established by \cite[Lemma 3.1]{benatti_nonlinear_2023} is in fact conditional on a suitable assumption concerning the asymptotic behavior of the level sets of $w_p$. In the next result, we show by means of \Cref{thm:asymptotic_behaviour} that asymptotic flatness is enough to obtain the monotonicity. 

\begin{proposition}\label{lem:comparison_bounded_sequence}
Let $(M,g)$ be an asymptotically flat Riemannian $3$-manifold with nonnegative scalar curvature. Let $w_p$ be the solution of \cref{eq:p-IMCF} issuing from $\Omega$ with connected boundary. The function $t\mapsto \tilde{\ma}_{H}^{(p)}(\Sigma^{(p)}_t)$ belongs to $W^{1,1}_{\mathrm{loc}}(0,+\infty)$ and is monotone nondecreasing. Moreover,
\begin{equation}\label{eq:comparison}
    \ma_H^{(p)}(\Sigma^{(p)}_t) \leq\tilde{\ma}_{H}^{(p)}(\Sigma^{(p)}_t)
\end{equation}
for almost every $t\geq0$, and
\begin{equation}\label{eq:same-limit}
    \lim_{t\to+\infty}\ma_H^{(p)}(\Sigma^{(p)}_t) = \lim_{t\to+\infty}\tilde{\ma}_{H}^{(p)}(\Sigma^{(p)}_t).
\end{equation}
\end{proposition}

\begin{proof}
Set
\begin{align}
N(t) &=\ncapa_p(\Sigma^{(p)}_t)^{-\frac1{p-1}} \left( 4\pi-\int_{\Sigma^{(p)}_t} \frac{\abs{\nabla w_p}^2}{(3-p)^2}\dif\sigma \right),&
D(t) &=\ncapa_p(\Sigma^{(p)}_t)^{-\frac{2}{(3-p)(p-1)}}.
\end{align}
Thus
\begin{equation}\label{eq:R-def}
    R(t)=\frac{N(t)}{D(t)}=4\pi(3-p)\,\tilde{\ma}_{H}^{(p)}(\Sigma^{(p)}_t).
\end{equation}
$N\in W^{1,1}_{\mathrm{loc}}(0,+\infty)$, while $D$ is smooth and strictly decreasing. More precisely, 
by \Cref{lem:expontial_growth} $D'(t) <0$ and $D(t) \to 0^+$ as $t \to +\infty$. Furthermore, computing the derivatives, we get
\begin{equation}\label{eq:Q-def}
    Q(t)=\frac{N'(t)}{D'(t)}=4\pi(3-p)\,\ma_H^{(p)}(\Sigma^{(p)}_t).
\end{equation}

\Cref{thm:asymptotic_behaviour} implies $N(t_j)\to 0$ as $j \to +\infty$. Fix $t\geq0$ and choose $j$ so large that $t_j>t$. Since $Q$ is nondecreasing, $Q(s)\geq Q(t)$ for almost every $s\in(t,t_j)$. Multiplying both sides by $D'(s)<0$, we get
\begin{equation}
    N'(s)=Q(s)D'(s)\leq Q(t)D'(s)
    \qquad\text{for a.e. }s\in(t,t_j).
\end{equation}
Integrating the differential inequality for $s \in (t,t_j)$, we obtain $N(t_j)-N(t)\leq Q(t) [D(t)-D(t_j)]$, thus 
\begin{equation}
    Q(t) \leq \frac{N(t_j)-N(t)}{D(t_j) - D(t)}
\end{equation}
since $D(t_j) < D(t)$. Taking now the limit as $j \to +\infty$ and recalling $D(t_j),N(t_j) \to 0$ as $j\to +\infty$, we find
\begin{equation}\label{eq:Q-le-R}
    Q(t) \leq \frac{N(t)}{D(t)} =R(t)
\end{equation}
In view of \eqref{eq:R-def} and \eqref{eq:Q-def}, this inequality is exactly \eqref{eq:comparison}. Consequently, for almost every $t$,
\begin{equation}
        R'(t)  =  \frac{N'(t)D(t)-N(t)D'(t)}{D(t)^2}   =   \frac{D'(t)}{D(t)}[Q(t)-R(t)]\geq0,
\end{equation}
because $D'(t)/D(t)<0$ and $Q(t)-R(t)\leq0$. Thus $t\mapsto R(t)$, and hence $t\mapsto\tilde{\ma}_{H}^{(p)}(\Sigma^{(p)}_t)$, is monotone nondecreasing. 

It remains to compare the limits. Since $Q$ is nondecreasing, the limit
\begin{equation}
    \ell=\lim_{t\to+\infty}Q(t)
\end{equation}
exists in $\mathbb R\cup\{+\infty\}$. If $\ell=+\infty$, then \eqref{eq:Q-le-R} immediately implies $R(t)\to+\infty$, so the two limits agree.

Assume now that $\ell<+\infty$. Since $Q(s)\leq\ell$ and $D'(s)<0$, we have
\begin{equation}
    N'(s)=Q(s)D'(s)\geq \ell D'(s)
    \qquad\text{for a.e. }s>t.   
\end{equation}
Integrating for $s\in(t,t_j)$ and dividing by the negative number $D(t_j)-D(t)$ gives
\begin{equation}\label{eq:secant-upper}
    \frac{N(t_j)-N(t)}{D(t_j)-D(t)}\leq\ell.
\end{equation}
Passing again to the limit along $t_j$ yields
\begin{equation}
    Q(t)\leq R(t)\leq\ell.
\end{equation}
Letting $t\to+\infty$ and using $Q(t)\to\ell$, we obtain
\begin{equation}
    \lim_{t\to+\infty}R(t)=\ell
    =\lim_{t\to+\infty}Q(t).
\end{equation}
Finally, equations \eqref{eq:R-def} and \eqref{eq:Q-def} give \eqref{eq:same-limit}.
\end{proof}

\begin{corollary}\label{lem:sequential_asymptotic_comparison}
    Under the assumptions of \Cref{lem:comparison_bounded_sequence}, one has
    \begin{equation}\label{eq:sequential_asymptotic_comparison}
        \lim_{t\to+\infty}\ma_H^{(p)}(\Sigma^{(p)}_t)\leq\liminf_{t\to+\infty}\ma_{\iso}^{(p)}(E^{(p)}_t).
    \end{equation}
\end{corollary}

\begin{proof}
    The proof is given in \cite[Lemma 3.2]{benatti_nonlinear_2023}, replacing \cite[Lemma 3.1]{benatti_nonlinear_2023} with \Cref{lem:comparison_bounded_sequence}.
\end{proof}

We are now going to use a similar technique to compare the $p$-Hawking mass with the Hawking mass up to replacing the area with the $p$-capacity. This result is also not entirely new; versions requiring stronger asymptotic assumptions have appeared in \cite{benatti_nonlinear_2023,benatti_isoperimetric_2025,agostiniani_riemannian_2025}. We show here that asymptotic flatness is enough.

\begin{proposition}\label{prop:p-Hawking-Hawking_comparison}
Let $(M,g)$ be an asymptotically flat Riemannian $3$-manifold with nonnegative scalar curvature.  Let $w_p$ be the solution of \cref{eq:p-IMCF} issuing from $\Omega$ with connected boundary. We have
\begin{equation}\label{eq:main-comparison}
\lim_{t\to+\infty}\ma_H^{(p)}(\Sigma^{(p)}_t)\leq\limsup_{t\to+\infty}\frac{\ncapa_p(\Sigma^{(p)}_t)^{\frac1{3-p}}}{8\pi}\left(4 \pi -  \int_{\Sigma^{(p)}_t} \frac{\H^2}{4} \dif \sigma\right).
\end{equation}
\end{proposition}

\begin{proof}
Set
\begin{align}
N(t) & = 4\pi+\int_{\Sigma^{(p)}_t}\frac{\abs{\nabla w_p}^2}{(3-p)^2}\dif\sigma-\int_{\Sigma^{(p)}_t}\frac{\abs{\nabla w_p}}{3-p}\H \dif\sigma& 
D(t)&=\ncapa_p(\Sigma^{(p)}_t)^{-\frac1{3-p}}.
\end{align}
Since $\ncapa_p(\Sigma_t)=e^t\ncapa_p(\Sigma)$, we have $D'(t)<0$ and $D(t) \to 0$. By \Cref{eq:p-Hawking}
\begin{equation}
\label{eq:mp-ratio}
8 \pi\,\ma_H^{(p)}(\Sigma^{(p)}_t)=\frac{N(t)}{D(t)}.
\end{equation}

We also introduce
\begin{align}
    F(t)&=\int_{\Sigma^{(p)}_t}\left(\frac{\abs{\nabla w_p}}{3-p}-\frac \H2\right)^2\dif \sigma,&
    G(t)&=4\pi-\int_{\Sigma^{(p)}_t}\frac{\H^2}{4}\dif \sigma. 
\end{align}
Completing the square yields the exact identity
\begin{equation}\label{eq:Phi-G-E}
    N(t)=G(t)+F(t).
\end{equation}

By the monotonicity formula
\begin{equation}\label{eq:monotonicity-J}
    \frac{\dd}{\dd t}\ma_H^{(p)}(\Sigma^{(p)}_t)\geq\frac{1}{8\pi(3-p)}\, \frac{Q(t)}{D(t)},\end{equation}
where
\begin{equation}
    Q(t)=4\pi-\int_{\Sigma^{(p)}_t}\frac{\sca^{\top}}2\dif\sigma+\int_{\Sigma^{(p)}_t}\frac{|\mathring \h|^2}{2}+\frac \sca2+\frac{|\nabla ^{\top}\abs{\nabla w_p}\vert^2}{\abs{\nabla w_p}^2}+\frac{5-p}{p-1}\left(\frac{\abs{\nabla w_p}}{3-p}-\frac \H2\right)^2\dif\sigma.\label{eq:J-def}
\end{equation}
All terms in the second integral are nonnegative. Moreover, by the Gauss--Bonnet theorem for $p$-IMCF in \Cref{prop:Gauss-Bonnet} we have
\begin{equation}
    \label{eq:gauss-bonnet}
    4\pi-\int_{\Sigma^{(p)}_t}\frac{\sca^{\top}}2\dif \sigma \geq0.
\end{equation}
Discarding all the nonnegative terms therefore gives
\begin{equation}\label{eq:J-ge-E}
Q(t)\geq \frac{5-p}{p-1} F(t)\geq F(t).
\end{equation}

Differentiating \cref{eq:mp-ratio}, we obtain
\begin{equation}\label{eq:mp-derivative}
    \frac{\dd}{\dd t}\ma_H^{(p)}(\Sigma_t)=\frac{1}{8\pi D(t)}\left[N'(t)+\frac{N(t)}{3-p}\right].
\end{equation}
Comparison with \cref{eq:monotonicity-J} yields
\begin{equation}
    \label{eq:Phi-prime}N'(t)\geq\frac{Q(t)-N(t)}{3-p}.
\end{equation}
Consequently,
\begin{align}
    \frac{N'(t)}{D'(t)}&\leq\frac{N(t)-Q(t)}{D(t)}\leq\frac{N(t)-F(t)}{D(t)}=\frac{G(t)}{D(t)},\label{eq:derivative-ratio}
\end{align}
where we first used~\eqref{eq:J-ge-E} and then~\eqref{eq:Phi-G-E}.

The differential inequality \cref{eq:derivative-ratio} converts the $p$-Hawking mass into the Hawking mass. Indeed, by \Cref{lem:asymptotic-gradient-level-sequence} there exists a divergent increasing sequence $(t_j)_{j \in \N}$ such that $N(t_j) \to 0$ as $j\to + \infty$. Fix $t\geq 0$ and choose $j$ such that $t_j>t$. Since $D'(s)<0$, \eqref{eq:derivative-ratio} gives
\begin{equation}
    N'(s) \geq \frac{G(s)}{D(s)}D'(s) \geq D'(s) \sup_{s \in [t, t_j]} \frac{G(s)}{D(s)}
\end{equation}
Integrating over $s\in (t, t_j)$, we get
\begin{equation}
    \frac{N(t_j)- N(t)}{D(t_j)-D(t)} \leq   \sup_{s \in [t, t_j]} \frac{G(s)}{D(s)}
\end{equation}
Since  $N(t_j), D(t_j) \to 0$ as $j \to +\infty$, we have
\begin{equation}
    \label{eq:tail-sup}
    8 \pi \ma_H^{(p)}(\Sigma^{(p)}_t) = \frac{N(t)}{D(t)}\leq\sup_{s\geq t}\frac{G(s)}{D(s)}= \sup_{s\geq t} \ncapa_p(\Sigma^{(p)}_s)^{\frac{1}{3-p}}\left[4 \pi -  \int_{\Sigma^{(p)}_s} \frac{\H^2}{4} \dif \sigma\right].
\end{equation}
Taking the limit as $t \to + \infty$ concludes the proof.
\end{proof}

\subsection{Proof of the main result}
Let $\Omega$ be an outward minimizing domain with connected boundary. Suppose now that almost every level $\Sigma^{(p)}_t$ is a $C^{1,1}$ surface. By \Cref{prop:p-Hawking-Hawking_comparison,prop:wellposedness} and \Cref{thm:asymptotic_behaviour}, we have
\begin{align}
    \lim_{t \to +\infty} \ma_H^{(p)}(\Sigma^{(p)}_t)&\leq\limsup_{t \to +\infty}\frac{\ncapa_p(\Sigma^{(p)}_t)^{\frac{1}{3-p}}}{8\pi}\left[4\pi-\int_{\Sigma^{(p)}_t}\frac{\H^2}{4}\dif\sigma\right]\\
    &\leq\limsup_{t \to +\infty}\frac{\ncapa_p(\Sigma^{(p)}_t)^{\frac{1}{3-p}}}{8\pi}\left[4\pi-\int_{\Sigma^{(p)}_t}\frac{\H^2}{4}+ \left(\frac{\H}{2}- \abs{\nabla w_2}\right)^2\dif\sigma\right]\\
    &\leq\limsup_{t \to +\infty}\frac{\ncapa_p(\Sigma^{(p)}_t)^{\frac{1}{3-p}}}{\ncapa_2(\Sigma^{(p)}_t)} \ma^{(2)}_H ( \Sigma^{(p)}_t) \leq \ma_{\mathrm{MY}}.
\end{align}
Moreover, \Cref{lem:p-capacity_to_area} yields
\begin{equation}
     \liminf_{p \to 1^+} \ma^{(p)}_{H}(\Sigma) \geq \ma_H(\Sigma),
\end{equation}
hence $ \ma_H(\Sigma)\leq\ma_{\mathrm{MY}}$. Therefore, applying \Cref{thm:jl-Hawking} with $\ma=\ma_{\mathrm{MY}}$, we conclude that
\begin{equation}
    \ma_{\iso}\leq\ma_{\mathrm{MY}},
\end{equation}
as desired.

The low regularity of $\Sigma^{(p)}_t$ does not allow us to proceed directly as before. We need to approximate $\Sigma^{(p)}_t$ with smooth surfaces. Here, we use the $\varepsilon$-regularization.

\begin{proof}[Proof of {\Cref{thm:intro-1}}]
    Let $\Omega$ be an outward minimizing domain with connected boundary. Let $t>0$, fix $T>t$ and choose a smooth bounded domain $D$ such that $E_T^{(p)}\Subset D$. Let $w_{p,\varepsilon}$ be the solution to \cref{eq:epsilon-regularized-p-imcf}. Again,
    \begin{equation}
        \frac{\ncapa_2(\Sigma^{(\varepsilon)}_t)}{8 \pi }\left[ 4\pi - \int_{\Sigma^{(\varepsilon)}_t} \frac{\H^2}{4}\dif \sigma\right] \leq \ma^{(2)}_H ( \Sigma^{(\varepsilon)}_t) \leq \ma_{\mathrm{MY}}
    \end{equation}
    by \Cref{prop:wellposedness} and \Cref{thm:Geroch-monotonicity}. By \cref{eq:convergence-q-cap_eps-regularized} and \Cref{prop:convergence-meancurvature}, we can find a vanishing subsequence $(\varepsilon_k)_{k \in \N} $ such that
    \begin{equation}
        \lim_{k\to+\infty} \ncapa_2(\Sigma^{(\varepsilon_k)}_t) \left[ 4 \pi - \int_{\Sigma^{(\varepsilon_k)}_t} \frac{\H^2}{4} \dif \sigma \right] =\ncapa_2(\Sigma^{(p)}_t) \left[ 4 \pi - \int_{\Sigma_t^{(p)}} \frac{\H^2}{4} \dif \sigma \right]
    \end{equation}
    for almost every $t \in (0,T)$. Since $T$ is arbitrary, for almost every $t \in (0,+\infty)$ we have
    \begin{equation}
        \frac{\ncapa_2(\Sigma_t^{(p)})}{8\pi} \left[ 4 \pi - \int_{\Sigma_t^{(p)}} \frac{\H^2}{4} \dif \sigma \right] \leq \ma_{\mathrm{MY}}.
    \end{equation}
    To conclude, \Cref{prop:p-Hawking-Hawking_comparison} and \cref{thm:asymptotic_behaviour} imply
    \begin{equation}
        \lim_{t \to +\infty} \ma_H^{(p)}(\Sigma_t^{(p)}) \leq \limsup_{t \to +\infty} \frac{\ncapa_p(\Sigma_t)^{\frac{1}{3-p}}}{\ncapa_2(\Sigma_t)}\ma_{\mathrm{MY}}\leq \ma_{\mathrm{MY}}.
    \end{equation}
    Using \Cref{thm:Geroch-monotonicity} and \Cref{lem:p-capacity_to_area}, we get 
    \begin{equation}
        \ma_H(\Sigma) \leq \liminf_{p\to 1^+} \ma^{(p)}_{H}(\Sigma) \leq \ma_{\mathrm{MY}}.
    \end{equation}
    Therefore, \Cref{thm:jl-Hawking} yields $\ma_{\iso} \leq \ma_{\mathrm{MY}}$.

    \smallskip

    On the other hand,  \Cref{thm:asymptotic_behaviour,prop:wellposedness,lem:sequential_asymptotic_comparison} yield
    \begin{equation}
        \ma_{\mathrm{MY}} = \lim_{t \to +\infty} \ma^{(2)}_H(\Sigma_t) \leq \limsup_{t \to +\infty} \ma^{(2)}_{\iso}(E_t)\leq \ma_{\iso}^{(2)}.
    \end{equation}
    Since $\ma^{(2)}_{\iso} =\ma_{\iso}$ by \Cref{thm:equivalence_luca}, we conclude the proof.
\end{proof}

\bigskip

We now specialize \cref{thm:intro-1} to asymptotically Schwarzschildian $3$-manifolds. These are asymptotically flat manifolds for which, in an asymptotically flat chart, the deviation $\zeta_{ij}=g_{ij}-\delta_{ij}$ from the Euclidean metric admits the expansion
\begin{equation}\label{eq:introduction-schwarzschild-asymptotics}
    \zeta_{ij}=\frac{2\ma}{\abs{x}}\delta_{ij}+o\left(\frac{1}{\abs{x}}\right)
\end{equation}
as $\abs{x}\to+\infty$, for some $\ma\in\R$.

\begin{corollary}\label{thm:intro-2}
    Let $(M,g)$ be an asymptotically Schwarzschildian, smooth, connected, complete Riemannian $3$-manifold with one end, and with compact, possibly empty, boundary. Assume that the scalar curvature of $(M,g)$ is nonnegative and that $\partial M$ is minimal. Suppose moreover that no other compact minimal surface is contained in $M$. 
   Then,
    \begin{equation}\label{eq:harmonic-mass-schwarzschild-parameter}
        \ma = \ma_{\mathrm{MY}}=\ma_{\iso}.
    \end{equation}
\end{corollary}

    \begin{proof}
    We only need to prove the first identity in \cref{eq:harmonic-mass-schwarzschild-parameter}.
        By \cref{eq:introduction-schwarzschild-asymptotics}, we have
        \begin{align}\label{eq:zzz-schw-identities}
            \delta^{ij}\zeta_{ij} = \frac{6\ma}{\abs{x}}+o(\abs{x}^{-1}) && \zeta_{ij} \frac{x^ix^j}{\abs{x}^2} = \frac{2\ma}{\abs{x}}+o(\abs{x}^{-1}).
        \end{align}
        Recall that $\abs{x}\geq r$ on $\An(r)$ and $\abs{\An(r)}_\delta \leq \kst r^3$. Therefore, substituting \cref{eq:zzz-schw-identities} into \cref{eq:local_mass_functional} and integrating, we obtain
        \begin{align}
            L^\psi_g(r)&=\frac{1}{r}\int_{\An(r)}\left(\frac{8\ma}{\abs{x}^2}\eta_\psi\left(\frac{\abs{x}}{r}\right)+\frac{4\ma}{r\abs{x}}\eta_\psi'\left(\frac{\abs{x}}{r}\right)\right)\dif x+o(1)\\
            &=16\pi \ma\int_1^4 2\eta_\psi(s)+s\eta_\psi'(s)\dif s+o(1)
        \end{align}
        Since $\eta_\psi$ is compactly supported in $(1,4)$, integration by parts gives
        \begin{equation}
            \int_1^4s\eta_\psi'(s)\dif s=-\int_1^4\eta_\psi(s)\dif s.
        \end{equation}
        Therefore,
        \begin{equation}\label{eq:local_mass_schwarzschild_limit}
             L^\psi_g(r)=16\pi \ma\int_1^4\eta_\psi(s)\dif s+o(1)= 3\pi \ma \int_0^1\frac{\psi(s)}{(1+s)^2}\dif s +o(1),
        \end{equation}
        from which we conclude the proof.
        \end{proof}

Finally, from \Cref{thm:intro-2} and the isoperimetric Riemannian Penrose inequality \cite[Theorem 1.3]{benatti_isoperimetric_2025}, the following holds.

\begin{corollary}\label{cor:Penrose_for_m}
    Let $(M,g)$ be a complete asymptotically Schwarzschildian Riemannian $3$-manifold of mass $\ma\in\R$, with nonnegative scalar curvature, and a possibly empty, smooth, closed, outermost minimal boundary. Suppose moreover that no other compact minimal surface is contained in $M$. Then,
    \begin{equation}
        \sqrt{\frac{\abs{N}}{16\pi}}\leq\ma
    \end{equation}
    for every connected component $N$ of $\partial M$. Equality holds if and only if $\partial M$ is connected and $(M,g)$ is isometric to a spatial Schwarzschild manifold of mass $\ma$.
\end{corollary}

\printbibliography
\end{document}